\documentclass[12pt]{article}
\usepackage[left=1.25in,right=1.25in,top=1.25in,bottom=1.25in]{geometry} 
\usepackage{amsmath}        
\usepackage{graphicx}
\usepackage{amssymb}
\usepackage{amsthm}
\usepackage{epstopdf}
\usepackage{amsfonts}
\usepackage[scr]{rsfso}
\usepackage{hyperref}
\usepackage{booktabs}
\usepackage{placeins}
\makeatletter
\def\amsbb{\use@mathgroup \M@U \symAMSb}
\makeatother
\usepackage{bbold}
\usepackage{natbib}
\usepackage{epigraph}
\usepackage{authblk}
\usepackage{subfig}

\newtheorem{thm}{Theorem}
\newtheorem{cor}{Corollary}
\newtheorem{de}{Definition}

\newtheorem{prop}{Proposition}
\def\E{\amsbb{E}}

\def\C{\amsbb{C}ov}
\def\P{\amsbb{P}}
\def\R{\amsbb{R}}
\def\N{\amsbb{N}}

\def\D{\mathrm{d}}
\def\Var{\amsbb{V}ar}
\def\C{\amsbb{C}ov}

\usepackage{setspace}

\DeclareMathOperator{\csch}{csch}
\DeclareMathOperator*{\erf}{erf}
\DeclareMathOperator*{\for}{\quad\text{for}\quad}

\newcommand\ind[1]{\amsbb{I}_{#1}}

\usepackage{color}
\definecolor{myblue}{rgb}{.8, .8, 1}

\usepackage{empheq}

\newlength\mytemplen
\newsavebox\mytempbox

\makeatletter
\newcommand\mybluebox{%
    \@ifnextchar[
       {\@mybluebox}%
       {\@mybluebox[0pt]}}

\def\@mybluebox[#1]{%
    \@ifnextchar[
       {\@@mybluebox[#1]}%
       {\@@mybluebox[#1][0pt]}}

\def\@@mybluebox[#1][#2]#3{
    \sbox\mytempbox{#3}%
    \mytemplen\ht\mytempbox
    \advance\mytemplen #1\relax
    \ht\mytempbox\mytemplen
    \mytemplen\dp\mytempbox
    \advance\mytemplen #2\relax
    \dp\mytempbox\mytemplen
    \colorbox{myblue}{\hspace{1em}\usebox{\mytempbox}\hspace{1em}}}

\makeatother

\title{The Jacobi theta distribution}

\author[1]{Caleb Deen Bastian\thanks{\href{mailto:cbastian@princeton.edu}{cbastian@princeton.edu}}}
\author[2,3]{Grzegorz A Rempala\thanks{G.R. acknowledges support from NSF Division of
Math Sciences grant 1853587}}
\author[1,4]{Herschel Rabitz\thanks{H.R. acknowledges support from US Army Research Office grant W911NF-19-1-0382}}
\affil[1]{Program in Applied and Computational Mathematics, Princeton University, Princeton, NJ. USA}
\affil[2]{Division of Biostatistics, The Ohio State University, Columbus, OH. USA}
\affil[3]{Department of Mathematics, The Ohio State University, Columbus, OH. USA}
\affil[4]{Department of Chemistry, Princeton University, Princeton, NJ. USA}

\date{\today}  

\begin{document}

\vspace{-6cm}                    
\maketitle

\begin{abstract}We form the Jacobi theta distribution through discrete integration of exponential random variables over an infinite inverse square law surface. It is continuous, supported on the positive reals, has a single positive parameter, is unimodal, positively skewed, and leptokurtic. Its cumulative distribution and density functions are expressed in terms of the Jacobi theta function.  We describe asymptotic and log-normal approximations, inference, and a few applications of such distributions to modeling.
\end{abstract}


\section{Introduction} We describe a univariate continuous distribution called the Jacobi theta distribution supported on the positive reals that does not appear in the literature to the best of our knowledge (see e.g., \cite{cud}). The distribution is attained from the action of an infinite random measure $N$ with random weights $\{W_x\}$ and fixed atoms $\N_{\ge1}$, where the weights are $iid$ exponential random variables with common mean $m\in(0,\infty)$, on the test function $f(x)=1/x^2$ for $x\in\N_{\ge1}$, such that $Nf$ is a random variable having the Jacobi theta distribution. It is represented as the infinite sum \[Nf = \sum_{x\ge1}W_x/x^2\] Its law is encoded in its Laplace transform \[\alpha\mapsto\sqrt{\alpha m}\pi\csch(\sqrt{\alpha m}\pi)\] The Jacobi theta distribution is continuous, has a single parameter $m$, is unimodal, positively skewed, and leptokurtic. 

This note is organized as follows. In Section~\ref{sec:back} we give the mathematical backdrop in terms of random measures. In Section~\ref{sec:density} we give the main result of the existence of the distribution and state some of its properties. In Section~\ref{sec:relation} we show that the distribution may be approximated by asymptotic expansion and by the log-normal distribution. In Section~\ref{sec:application} we give three applications to modeling data. In Section~\ref{sec:conc} we end with discussions and conclusions.

\section{Background}\label{sec:back} We give background using notation and conventions from \cite{cinlar,rm}.
Let $(\Omega,\mathscr{H},\P)$ be a probability space and let $(E,\mathscr{E})$ be a measurable space. A \emph{random measure} is a transition kernel from $(\Omega,\mathscr{H})$ into $(E,\mathscr{E})$. Specifically the mapping $N:\Omega\times E\mapsto\R_{\ge0}$ is a random measure if $\omega\mapsto N(\omega,A)$ is a random variable for each $A$ in $\mathscr{E}$ and if $A\mapsto N(\omega,A)$ is a measure on $(E,\mathscr{E})$ for each $\omega$ in $\Omega$. We denote $\mathscr{E}_{\ge0}$ the set of non-negative $\mathscr{E}$-measurable functions. 

The law of $N$ is uniquely determined by the \emph{Laplace functional} $L$ from $\mathscr{E}_{\ge0}$ into $[0,1]$ \begin{equation} L(f)=\E e^{-Nf} = \E\exp_-\int_EN(\D x)f(x)\for f\in\mathscr{E}_{\ge0}\end{equation} The Laplace functional encodes all the information of $N$: its distribution, moments, etc. The distribution of $Nf$, denoted by $\eta$, i.e. $\eta(\D x)=\P(Nf\in \D x)$, is encoded by the Laplace transform, which may be expressed in terms of the Laplace functional \begin{equation}F(\alpha) = \E e^{-\alpha Nf} =  \E e^{-N(\alpha f)}=L(\alpha f)\for \alpha\in\R_{\ge0}\end{equation} The \emph{moments} of $Nf$ (if they exist) can be attained from the Laplace functional \begin{equation}\label{eq:moments} \E(Nf)^n = (-1)^n\lim_{q\downarrow0}\frac{\partial^n}{\partial q^n}L(qf)=\lim_{q\downarrow0}\int_0^\infty \eta(\D x)x^ne^{-qx}\quad\text{for}\quad n\in\N_{\ge1}\end{equation} 

Let $D\subset E$ be a countable subset of $E$ and let $\{W_x: x\in D\}$ be an independency of non-negative random variables distributed $W_x\sim\nu_x$ with mean $m_x$ and variance $\sigma_x^2$. The random measure $N$ on $(E,\mathscr{E})$ formed as \begin{equation}\label{eq:fixed}N(A) = \int_{E}N(\D x)\ind{A}(x) = \sum_{x\in D}W_x\ind{A}(x)\for A\in\mathscr{E}\end{equation} is \emph{additive} with \emph{fixed atoms} of $D$ and \emph{random weights} of $\{W_x\}$. The Laplace functional of $N$ is given by \begin{equation}\label{eq:fixedlaplace}L(f)=\E e^{-Nf} = \prod_{x\in D}\int_{\R_{\ge0}}\nu_x(\D z)e^{-zf(x)} = \prod_{x\in D}F_x(f(x))\for f\in\mathscr{E}_{\ge0}\end{equation} where $F_x$ is the Laplace transform of $\nu_x$ defined as \begin{equation}F_x(\alpha) = \int_{\R_{\ge0}}\nu_x(\D z)e^{-\alpha z}\for \alpha\in\R_{\ge0}\end{equation} The Laplace transform $F$ of $Nf$ is expressed in terms of the Laplace functional $F(\bullet)=L(\bullet f)$. 

$Nf$ is formed as \begin{equation}\label{eq:Nf}Nf = \int_EN(\D x)f(x) = \sum_{x\in D}W_xf(x)\for f\in\mathscr{E}_{\ge0}\end{equation} and has mean, variance, and second moment \begin{align*}\E Nf &= \sum_{x\in D}m_xf(x)\\\Var Nf &= \sum_{x\in D}\sigma_x^2f^2(x)\\\E(Nf)^2&=\Var Nf + (\E Nf)^2\end{align*} and covariance of $f,g\in\mathscr{E}_{\ge0}$ \[\C(Nf,Ng) = \sum_{x\in D}\sigma_x^2f(x)g(x)\] Hence, assuming at least one of the $\{W_x\}$ is non-degenerate, the covariance is zero if and only if the functions are disjoint.

 A product random measure $M=N\times N$ on $(E\times E,\mathscr{E}\otimes\mathscr{E})$ can be defined as \begin{align*}\label{eq:product}Mf &=\int_{E\times E}M(\D x,\D y)f(x,y)\nonumber\\&=\int_{E\times E}N(\D x)N(\D y)f(x,y)\nonumber\\&= \sum_{(x,y)\in D^2} W_xW_yf(x,y)\for f\in(\mathscr{E}\otimes\mathscr{E})_{\ge0}\end{align*} with Laplace functional \[\E e^{-Mf} = \prod_{(x,y)\in D^2}F_{xy}(f(x,y))\for f\in(\mathscr{E}\otimes\mathscr{E})_{\ge0}\] where $F_{xy}=F_x\circledast F_y$ is the Laplace transform of $W_xW_y$ and $\circledast$ is the convolution operator. For product functions $f=g\times g\in(\mathscr{E}\otimes\mathscr{E})_{\ge0}$, we have that $Mf=(Ng)^2$. This readily extends to $n$-products with $M=\bigtimes_x^n N$ for $f=\bigtimes_{x}^n g$, where we have $Mf = (Ng)^n$. Therefore \begin{equation}\label{eq:Nn}\E(Ng)^n = \sum_{(x,\dotsb,y)\in D^n}\E(W_x\dotsb W_y)g(x)\dotsb g(y)\for g\in\mathscr{E}_{\ge0},\quad n\ge1\end{equation}

\section{Distribution}\label{sec:density} 


We define the Jacobi theta function.

\begin{de}[Jacobi theta function]\label{re:jacobi} The Jacobi theta function is defined as \begin{equation}\label{eq:jacobi}\theta_2(z,q)=2q^{1/4}\sum_{k=0}^\infty q^{k(k+1)}\cos((2k+1)z)\end{equation}
\end{de}

Now we give the main result on the existence of the Jacobi theta distribution.

\begin{samepage}
\begin{thm}[Jacobi theta]\label{thm:csch} Consider the random measure $N$ \eqref{eq:fixed}. Let $D=\N_{\ge1}$ with $(E,\mathscr{E})=(\R_{>0},\mathscr{B}_{\R_{>0}})$ and let $\{W_x\}$ be an independency of exponential random variables with common mean $m\in(0,\infty)$. Then for $f\in\mathscr{E}_{\ge0}$ as $f(x)=1/x^2$, $Nf$ has Laplace transform \[F(\alpha) = \sqrt{\alpha m}\pi\csch(\sqrt{\alpha m}\pi)\for\alpha\in\R_{\ge0}\] cumulative distribution function \[\eta(Nf\le x)= \sqrt{\frac{m\pi}{x}}\theta _2\left(0,e^{-\frac{m \pi ^2}{x}}\right)\] and density \[\eta(\D x) = \frac{\sqrt{m\pi}}{2x^{5/2}}\left(2m\pi^2\left(\frac{\theta _2\left(0,e^{-\frac{m \pi ^2}{x}}\right)}{4}+2e^{-\frac{m\pi^2}{x}}\sum_{k\ge1}k(k+1)(e^{-\frac{m\pi^2}{x}})^{k(k+1)-3/4}\right)-x\theta _2\left(0,e^{-\frac{m \pi ^2}{x}}\right)\right)\D x\] \end{thm}
\end{samepage}
\begin{proof} The Laplace transform is computed as \begin{align*}\E e^{-\alpha Nf} &=\E e^{-\sum_{x\in D}\alpha W_xf(x)}\\&=\prod_{x\in D}\E e^{-\alpha W_xf(x)}\\&=\prod_{x\in D}\int_0^\infty\D z \frac{e^{-z/m_x}}{m_x}e^{-\alpha z f(x)}\\&=\prod_{x\in D}\frac{1}{1+\alpha m_xf(x)}\end{align*} which specialized for $m_x=m$ gives the result. The inverse Laplace transform follows from  noting that \begin{align*}L^{-1}(\frac{\csch{\sqrt{a}}}{\sqrt{a}})(x)&=\frac{2}{\sqrt{\pi x}}\sum_{k\ge 0}e^{-(2k+1)^2/(4x)}\\&=\frac{\theta_2(0,e^{-1/x})}{\sqrt{\pi x}}\end{align*} so that \[L^{-1}(\sqrt{a}\csch(\sqrt{a}))(x)=\frac{\partial}{\partial x}L^{-1}(\frac{\csch{\sqrt{a}}}{\sqrt{a}})(x)\] giving the cumulative distribution function and density upon substitution of $x\gets x/(m\pi^2)$. The derivative follows from the derivative of $\theta_2$ in the second coordinate. 
\end{proof}

The frequency spectrum of $Nf$ is given by $C(\omega)=F(-i\omega)$ for $\omega\in\R$ with magnitude squared \[|C(\omega)|^2 = |\omega| m\pi^2\csch(\sqrt{-i\omega m}\pi)\csch(\sqrt{i\omega m}\pi)\for \omega\in\R\] and the phase spectrum is given by $P(\omega)=\arctan(\mathfrak{I}(C(\omega))/\mathfrak{R}(C(\omega)))$. Both are plotted below in Figure~\ref{fig:cschft} for $m=7$.

\begin{figure}[h!]
\centering
\begingroup
\captionsetup[subfigure]{width=3in,font=normalsize}
\subfloat[Frequency spectrum\label{fig:freq}]{\includegraphics[width=3.5in]{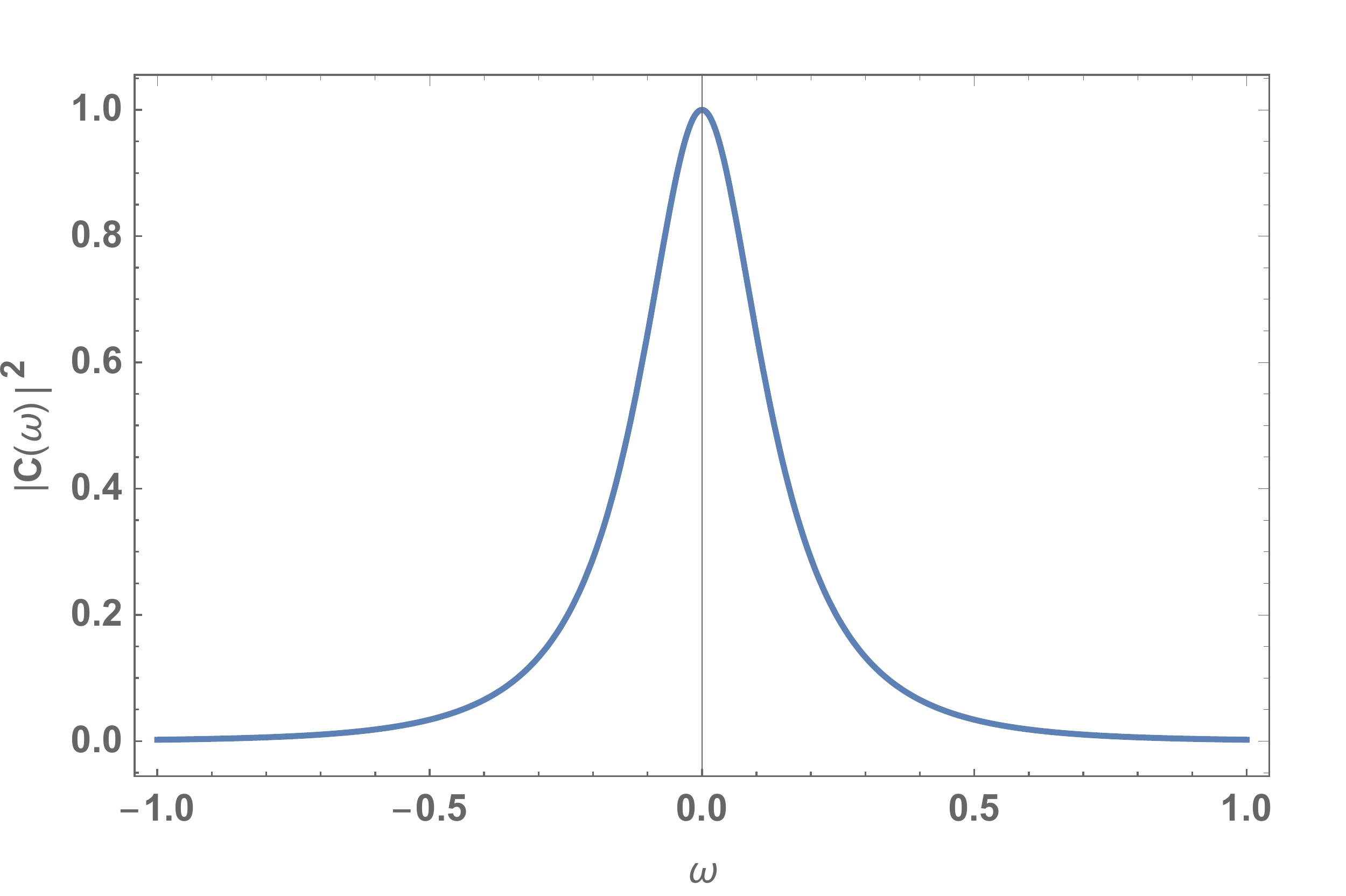}}\\
\subfloat[Phase spectrum\label{fig:phase}]{\includegraphics[width=3.5in]{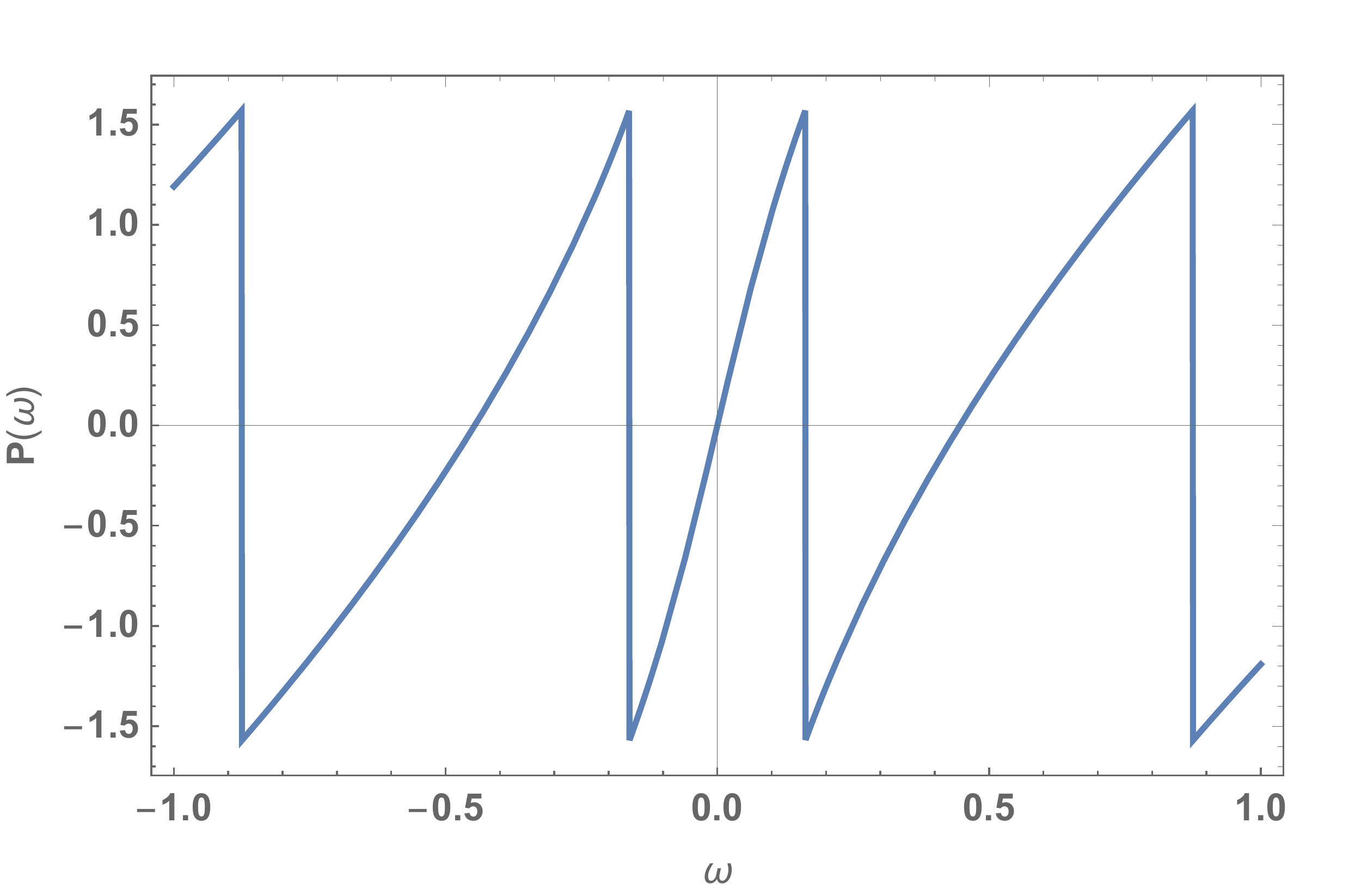}}\\
\endgroup
\caption{Frequency and phase spectra  $Nf$ for $m=7$}\label{fig:cschft}
\end{figure}

\FloatBarrier

The statistics follow and are expressed in terms of the Riemann zeta function. Both skewness and kurtosis are constant.
\begin{samepage}
\begin{cor}[Statistics] The mean, variance, and second moment are \begin{align*}\E Nf &=m\zeta(2)= m\pi^2/6\\\Var Nf &= m^2\zeta(4)=m^2\pi^4/90\\\E(Nf)^2&=7m^2\pi^4/180 \end{align*} with constant signal to noise ratio $\E Nf/\sqrt{\Var Nf}=\sqrt{5/2}\simeq1.58$, where $\zeta$ is the Riemann zeta function. The skewness is \[\text{Skewness}(Nf) = 4\sqrt{10}/7\simeq 1.81\] and the kurtosis is \[\text{Kurtosis}(Nf)=57/7\simeq8.14\] 
\end{cor}
\end{samepage}

%

Next we show that the moments are finite at all orders and that the Jacobi theta distribution is uniquely defined by its moments.

\begin{samepage}
\begin{prop}[Finite characterizing moments]\label{prop:finite} $Nf$ has finite moments of all orders \[\E (Nf)^n <\infty\for m>0,\quad n\ge1\] with moment generating function \[M(t)=F(t)<\infty\for t\in\R\] where $F$ is the Laplace transform of $Nf$.
\end{prop}
\end{samepage}
\begin{proof}The finiteness of all moments follows from \eqref{eq:Nn}, where \begin{align*}\E(Nf)^n &= \sum_{(x,\dotsb,y)\in D^n}\E(W_x\dotsb W_y)f(x)\dotsb f(y)\\&\le \sum_{(x,\dotsb,y)\in D^n}\E W_x^n f(x)\dotsb f(y)\\&=n!(m\zeta(2))^n<\infty\for m>0,\quad n\ge1\end{align*} For the mgf, we have that $M(t)=F(-t)=F(t)$ for $t\in\R$.
\end{proof}

Below in Figure~\ref{fig:csch} we show the densities for $m=1,\dotsb,7$. 
\begin{figure}[h!]
\centering
\includegraphics[width=6in]{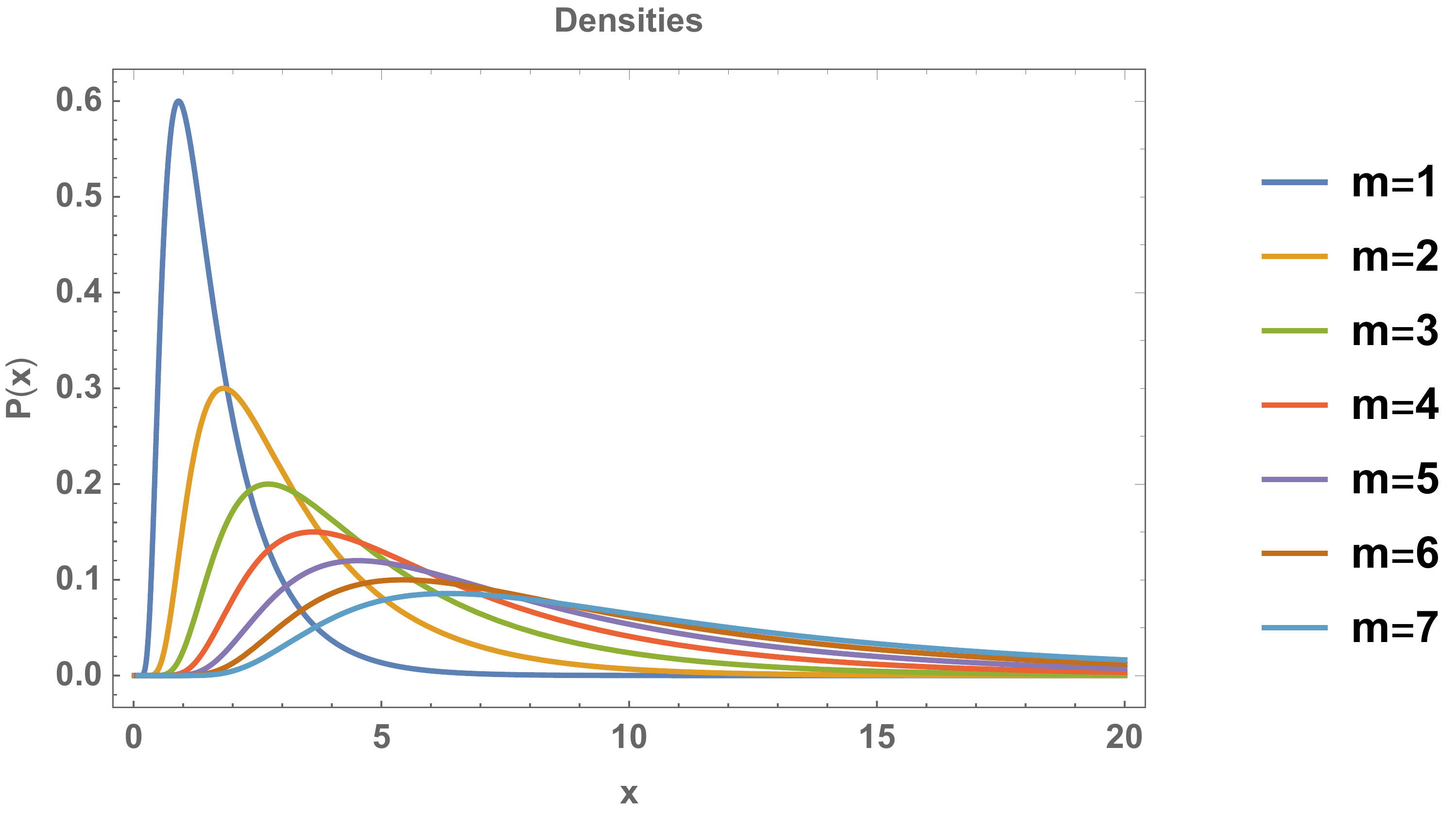}
\caption{Density $\eta$ of Jacobi theta distribution for $m=1,\dotsb,7$}\label{fig:csch}
\end{figure}
\FloatBarrier

The CDF may be used to furnish an estimator for $m$ given data. 

\begin{samepage}
\begin{prop}[Estimator] Let $\{X_i=1,\dotsb,n\}$ be an independency of data with empirical distribution function $F_n$. Let $u\in(0,1)$ be a cdf value at the point $x=F_n^{-1}(u)$, e.g., $u=1/2$. Then $m$ may be estimated by numerically finding the root of the following equation \[\sqrt{\frac{m\pi}{x}}\theta _2\left(0,e^{-\frac{m \pi ^2}{x}}\right)=u\] 
\end{prop}
\end{samepage}

\section{Approximations}\label{sec:relation} 

We discuss how the Jacobi theta distribution may be approximated through asymptotics and the log-normal distribution. 

\begin{samepage}
\subsection{Asymptotic} \begin{prop}[First-order] The Jacobi theta distribution cumulative distribution function may be approximated to first-order as \[\P(Nf\le x)\simeq 2\sqrt{\frac{m\pi}{x}}e^{-\frac{m \pi ^2}{4x}}\for x\le m\pi^2/2\] with total mass $\P(Nf\le m\pi^2/2)=2\sqrt{\frac{2}{e\pi}}\simeq 0.968$ and density function \[\eta(\D x)\simeq \frac{\sqrt{\pi } e^{-\frac{\pi ^2 m}{4 x}} \left(\pi ^2 m-2 x\right) \sqrt{\frac{m}{x}}}{2 x^2}\D x\for x\in(0,m\pi^2/2]\]
\end{prop}
\end{samepage}
\begin{proof} The proof follows from the first-order truncation of representation \eqref{eq:jacobi}.
\end{proof}

Higher-order approximations are admitted in view of \eqref{eq:jacobi}.

We show the cumulative distribution function of the Jacobi theta distribution for $m=7$ against the asymptotic approximation on the given support.

\begin{figure}[h!]
\centering
\includegraphics[width=6in]{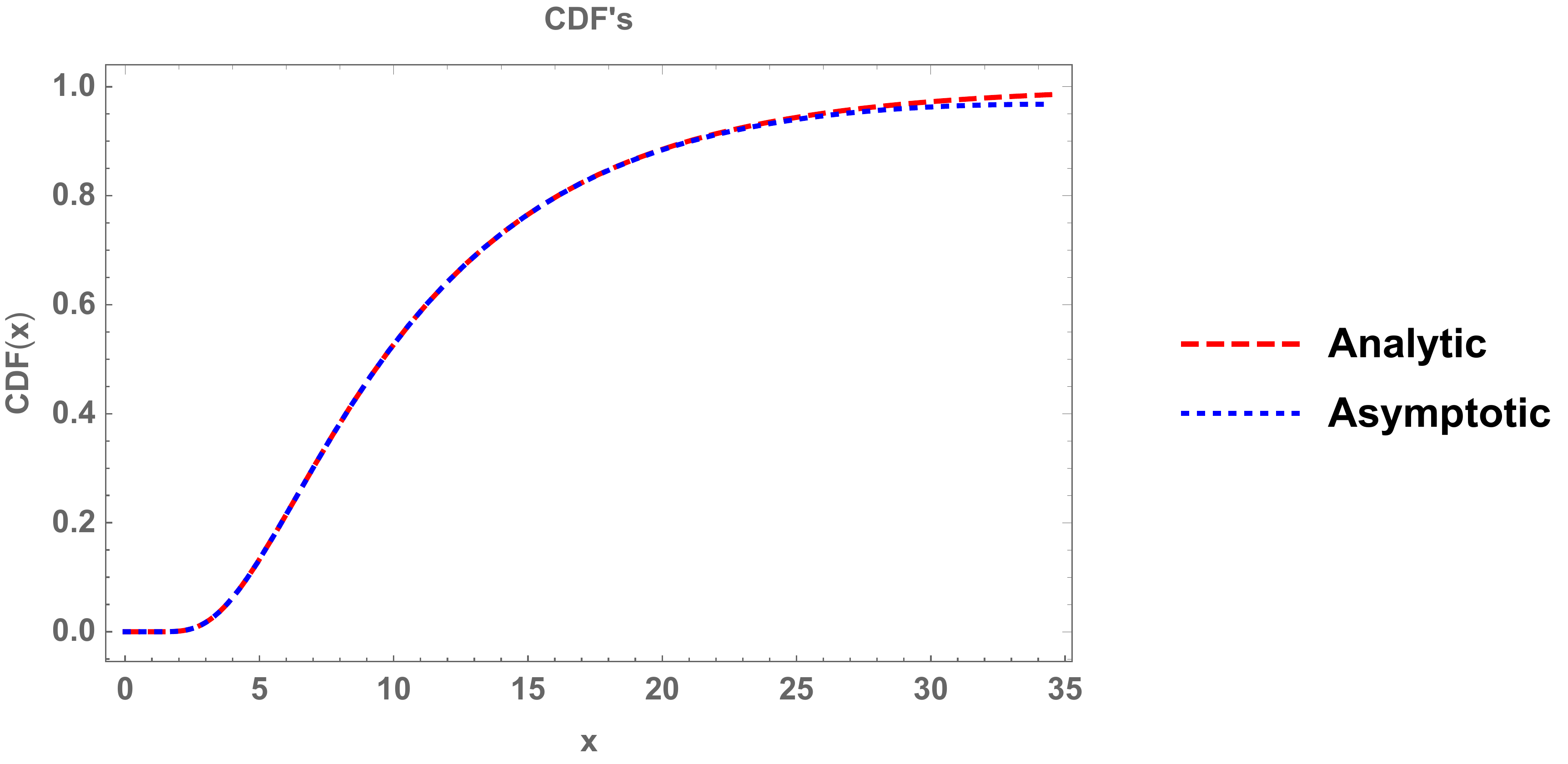}
\caption{Cumulative distribution function of Jacobi theta distribution for $m=7$ and asymptotic approximation}\label{fig:csch2}
\end{figure}
\FloatBarrier

The asymptotic approximation may be used to furnish an estimator for $m$ given data. 

\begin{samepage}
\begin{prop}[Estimator] Let $\{X_i=1,\dotsb,n\}$ be an independency of Jacobi theta random variables with empirical distribution function $F_n$ converging almost surely to $F$. Let $u\in(0,2\sqrt{\frac{2}{e\pi}}]\subset(0,1)$ be an argument to $F^{-1}$ such that $F^{-1}(u)$ is closely approximated by $F_n^{-1}(u)$, e.g., $u=1/2$. Then $m$ is estimated as  \[m_n=F_n^{-1}(u)\left(\frac{-2W_{-1}\left(-\pi  u^2/8\right)}{\pi ^2}\right)\] where $W$ is the product-logarithm function. 
\end{prop}
\end{samepage}
\begin{proof} We solve the equation $2\sqrt{\frac{m_n\pi}{x_n}}e^{-\frac{m_n \pi ^2}{4x_n}}=u$ for $m_n$, where $x_n=F_n^{-1}(u)$. 

\end{proof}
%

\subsection{Log-normal} The Jacobi theta distribution can be approximated by the log-normal distribution $\text{LogNormal}(\mu,\sigma)$.  Matching first and second moments, we have \begin{align}\mu&=\log \left(\frac{1}{6} \sqrt{\frac{5}{7}} \pi ^2 m\right)\label{eq:first}\\\sigma&=\sqrt{\log \left(\frac{7}{5}\right)}\label{eq:second}\end{align} which gives approximate density \[\widetilde{\eta}(\D x) = \frac{\exp \left(-\frac{\left(\log \left(\frac{1}{6} \sqrt{\frac{5}{7}} \pi ^2 m\right)-\log (x)\right)^2}{2 \log \left(\frac{7}{5}\right)}\right)}{x \sqrt{2 \pi  \log \left(\frac{7}{5}\right)}}\D x\] and cumulative distribution function \[\widetilde{\eta}(t\le x)=\frac{1}{2}(1+\erf\left(\frac{\log x - \mu}{\sigma\sqrt{2}}\right))\] The skewness is $17\sqrt{2/5}/5\simeq2.15$ and the kurtosis is $7\,631/625\simeq 12.21$, both independent of $m$.

The entropy can be approximated as \[H\simeq \log_2(\sigma\sqrt{2\pi}\exp(\mu+1/2))=\log_2 \left(\frac{1}{3} \sqrt{\frac{5 e}{7}} \pi ^{5/2} m \sqrt{\log \left(\sqrt{\frac{7}{5}}\right)}\right)\sim O(\log_2(m))\]

In Figure~\ref{fig:csch2} we show the cumulative distribution function of the Jacobi theta distribution for $m=7$ against the matching log-normal distribution. The approximation is very good.


\begin{figure}[h!]
\centering
\includegraphics[width=6in]{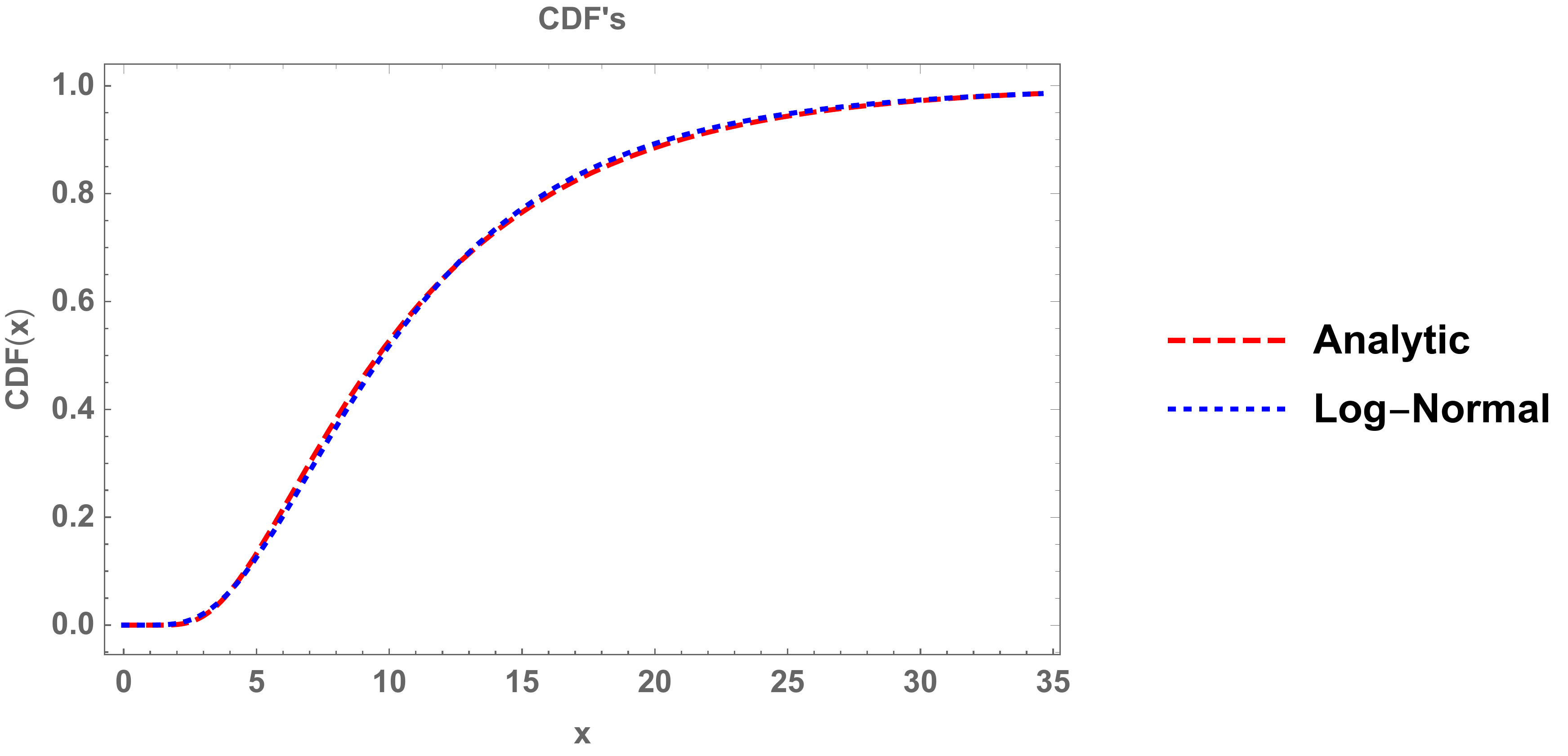}
\caption{Cumulative distribution function of Jacobi theta distribution for $m=7$ and log-normal approximation}\label{fig:csch2}
\end{figure}
\FloatBarrier

\subsection{Maximum likelihood}

\begin{prop}[MLE] Let $\mathbf{X}=\{X_i:i=1,\dotsb,n\}$ be a collection of data. Then the log-normal maximum likelihood estimator with parameters $\mu$ \eqref{eq:first} and $\sigma$ \eqref{eq:second} for $m\in(0,\infty)$ of the Jacobi theta distribution is given by \[\hat{m}= \frac{6}{\pi^2}\sqrt{\frac{7}{5}}\left(\prod_{i}^nX_i\right)^{1/n}\sim\text{LogNormal}(\mu-\sigma^2/(2n)+\log\left(\frac{6}{\pi^2}\sqrt{\frac{7}{5}}\right),\sigma/\sqrt{n})\] with mean and variance \begin{align*}\E(\hat{m})&=\frac{6}{\pi^2}\sqrt{\frac{7}{5}} \exp\mu=m\\\Var(\hat{m})&=m^2\left(\exp(\frac{\sigma ^2}{n})-1\right)\xrightarrow{n \text{ large}}m^2\left(\frac{\sigma ^2}{n}\right)\end{align*}
\end{prop}
\begin{proof} The parameter $m$ can be estimated using the log-normal maximum likelihood estimator of $\mu$ \[\hat{\mu}=\frac{\sum_i^n \log(X_i)}{n}\xrightarrow{n \text{ large}}\text{Gaussian}(\mu,\sigma^2/n)\] The log-normal distribution follows from the MLE of the log-normal distribution. 
\end{proof}


\section{Applications}\label{sec:application} The Jacobi theta distribution can be applied to any problem where the random variable of interest is an infinite superposition of exponential random variables with inverse quadratically scaled means relative to the positive integers. Superpositions at integer points correspond to discrete integration over an inverse square law surface.  Inverse square laws are found in many applications, including gravitation, electric fields and forces, intensity of light, radiation from a source, intensity of sound, and so on. Exponential distributions arise as the maximum entropy distribution on the positive reals with fixed mean.

\subsection{Maximum likelihood estimation} Consider a random sample of data $\mathbf{X}=\{X_i:i=1,\dotsb,n\}$ distributed according to the Jacobi theta distribution with parameter $m\in(0,\infty)$. We estimate the parameter $m$ using the exact and asymptotic CDF methods and the log-normal approximation. We take $m=7$, $n=10^2$ and generate $10^4$ samples of $\mathbf{X}$ as $\{\mathbf{X}^j:j=1,\dotsb,10^4\}$. For each $\mathbf{X}^j$ we estimate $m$ using the three estimators. We show the histogram distributions of the estimators below in Figure~\ref{fig:estimators}. The log-normal estimator has the smallest variance.

\begin{figure}[h!]
\centering
\includegraphics[width=6in]{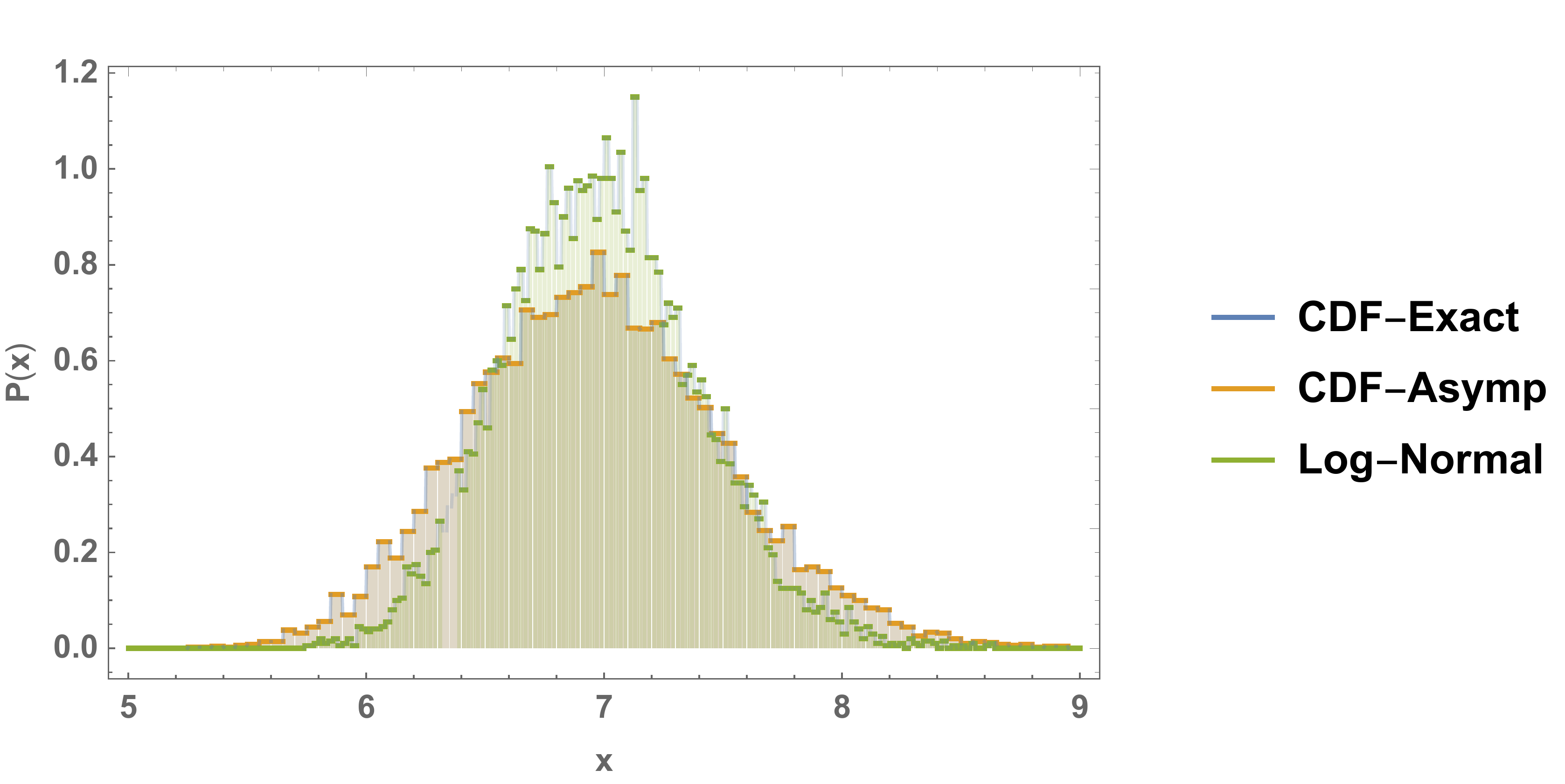}
\caption{Histogram distributions of estimators for $m$ given 100 bins}\label{fig:estimators}
\end{figure}
\FloatBarrier


\subsection{Radio-frequency interference field} Consider an equispaced grid of locations of interferers transmitting radio-frequency waves with constant radial spacing $d>0$. The transmitted powers are denoted $\{P_{x}: x=1,2,\dotsb\}$ and are distributed $\text{Exponential}(4\pi d^2 \lambda)$ where $\lambda>0$. The sphere's surface area is $4\pi r^2$ and the power at the origin from interferer $x$ is $I_x=P_x/r^2$. Therefore, the total power received at the origin $I$ from the field of interferers is Jacobi theta distributed with $m=1/(4\pi d^2\lambda)$, having mean \[\E I = \E\sum_x I_x = \frac{\zeta(2)}{4\pi\lambda d^2}=\frac{\pi}{24\lambda d^2}\] and variance \[\Var I = \frac{\zeta(4)}{16\pi^2\lambda^2 d^4}= \frac{\pi^2}{1\,440 \lambda^2d^4}\] In Figure~\ref{fig:rf} we show $z=10^3$ random locations on a disk of radius $z$ with unit radial spacing. The distribution of points on the plane concentrates towards the origin, with point density \[\frac{1}{\pi r}\sim O(r^{-1})\]

\begin{figure}[h!]
\centering
\includegraphics[width=5in]{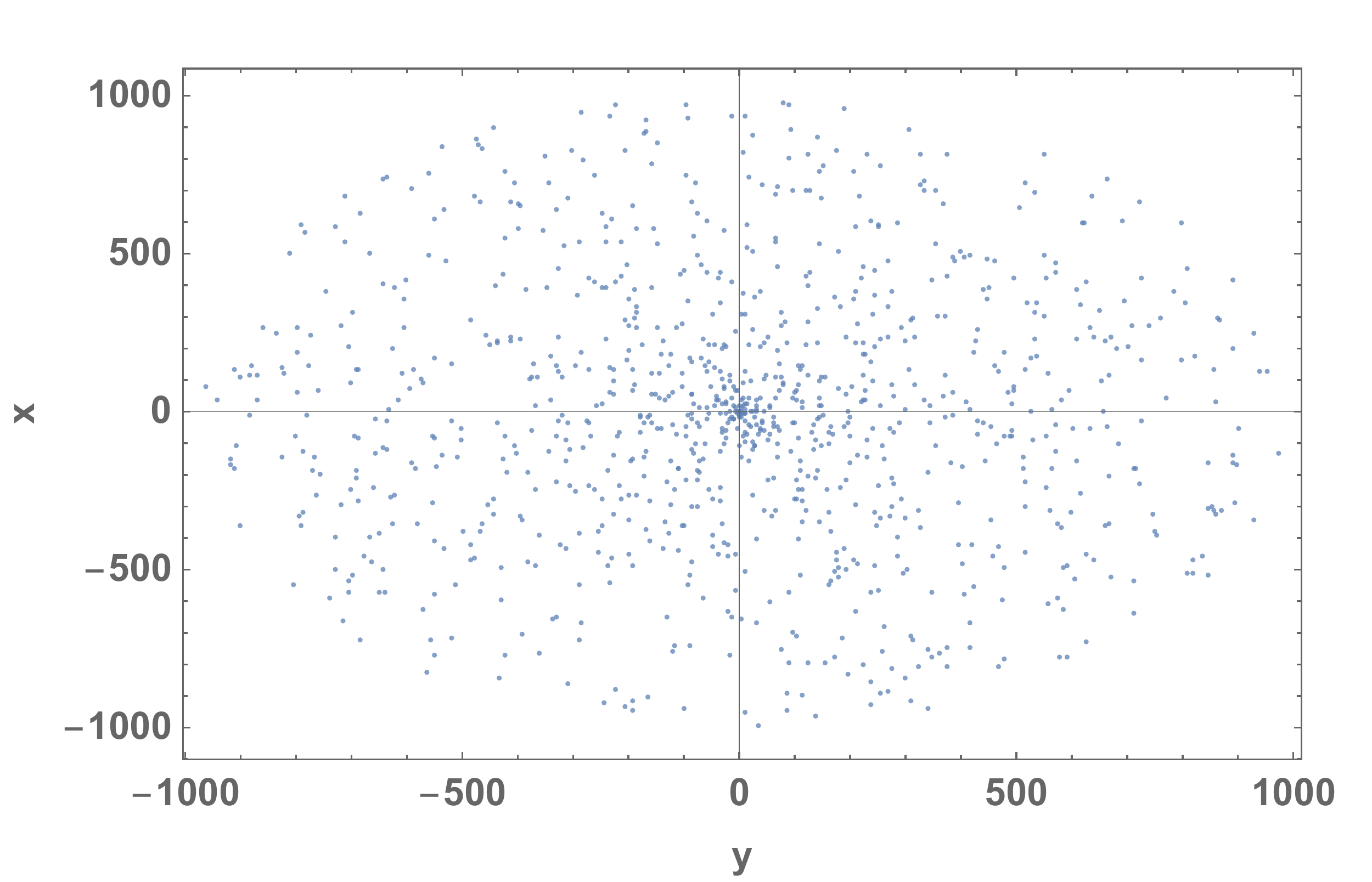}
\caption{$10^3$ random locations with constant radial spacing}\label{fig:rf}
\end{figure}
\FloatBarrier 

Next we slightly alter the problem by changing the spacing of the points and restricting to $[0,t]$.  Altering the points changes the distribution, yet the mean and variance may be readily attained. Consider the grid on $[0,t]$ formed by $D_t=\{\sqrt{tx}:x=1,\dotsb,t\}$. Then $I$ on $[0,t]$ is denoted $I_{t}$ \[I_t = \sum_{z=\sqrt{tx}\in D_t} I_z \] with mean and variance \begin{align*}\E I_t &=\sum_{z=\sqrt{tx}\in D_t}\frac{1}{4\pi\lambda d^2z^2}=\sum_{z=\sqrt{tx}\in D_t}\frac{1}{4\pi\lambda d^2tx}=\frac{H_t/t}{4\pi\lambda d^2}\\\Var I_t &=\sum_{z=\sqrt{tx}\in D_t}\frac{1}{(4\pi\lambda d^2z^2)^2}=\sum_{z=\sqrt{tx}\in D_t}\frac{1}{(4\pi\lambda td^2x)^2}=\frac{H_t^{(2)}/t^2}{(4\pi\lambda d^2)^2}\end{align*} where $H_t$ is the $t$-th Harmonic number. Note that $H_t/t<\zeta(2)$ for $t\ge1$ and thus the interference at the origin is reduced with the altered spacing. The distribution of points on the plane is constant, with point density $1/\pi$. We show $10^3$ random locations with linear radial spacing below in Figure~\ref{fig:rf2}.

\begin{figure}[h!]
\centering
\includegraphics[width=5in]{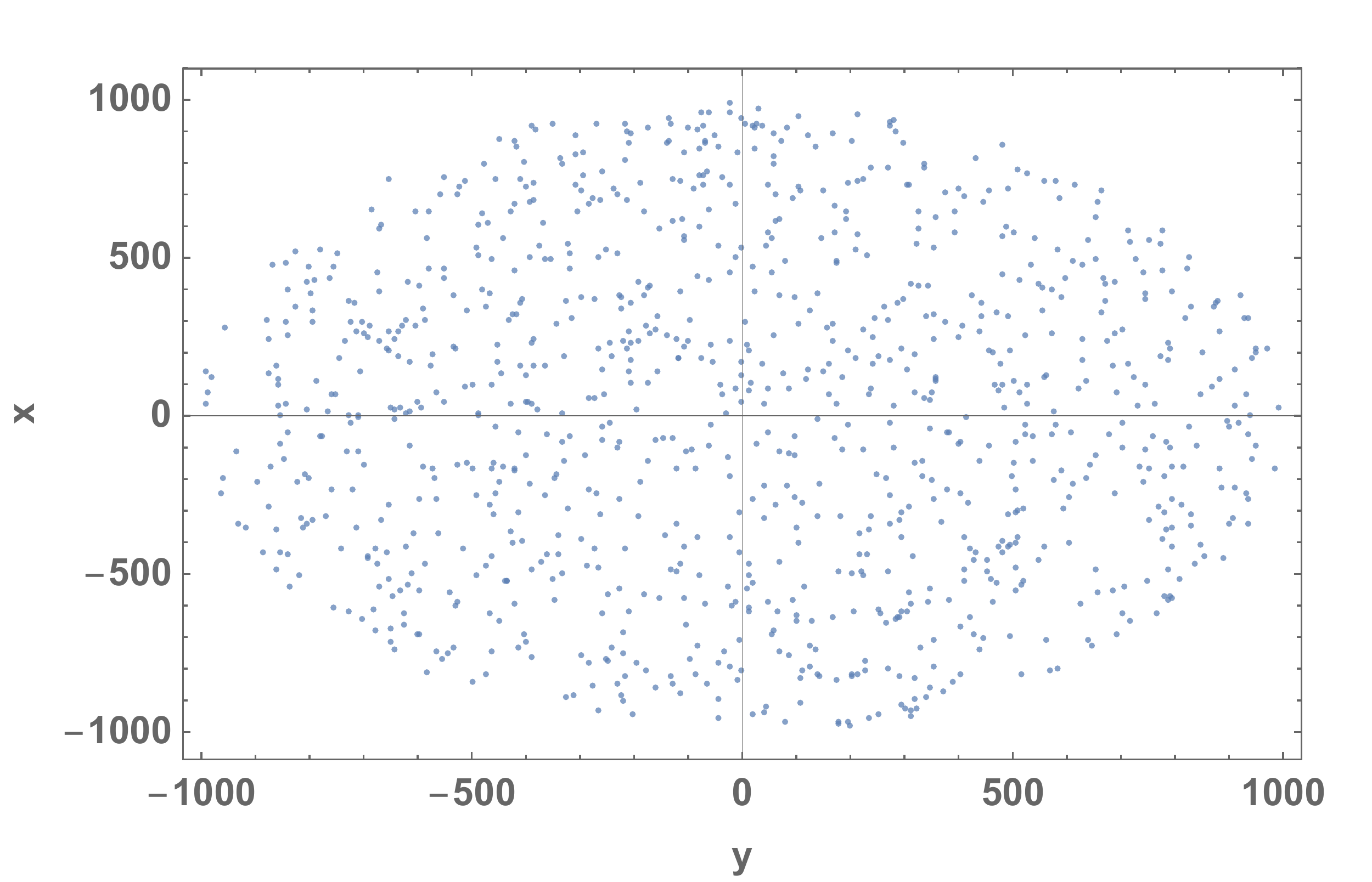}
\caption{$10^3$ random locations with linear radial spacing}\label{fig:rf2}
\end{figure}
\FloatBarrier 

Suppose we have a transmitter at some location $z$ with power $P_z\sim\text{Exponential}(4\pi d^2 \lambda z)$ superimposed with interfering field $I$ comprised of interferers  $\{P_x\}$ as before, with Jacobi theta distribution with parameter $m=1/(4\pi d^2\lambda)$. We are interested in the mean and variance of the random variable of $Q_z=P_z/I$---the signal interference noise ratio (SINR)---which may be calculated through the ratio distribution using the log-normal approximation as \begin{align*}\E Q_z &\simeq\frac{21}{10 \pi ^4 d^2 \lambda  m z}\sim O(z^{-1})\\\Var Q_z &\simeq\frac{3969}{500 \pi ^8 d^4 \lambda ^2 m^2 z^2}\sim O(z^{-2})\end{align*} and $\E Q_z/\sqrt{\Var Q_z}\simeq\sqrt{5}/3\simeq 0.75$. We plot $\P(Q_z>t)$---the coverage probability---for $z=1$, $m\in\{3,5,7,9\}$, $\lambda=1/100$, and $d=1$ below in Figure~\ref{fig:sinr}.

\begin{figure}[h!]
\centering
\includegraphics[width=5in]{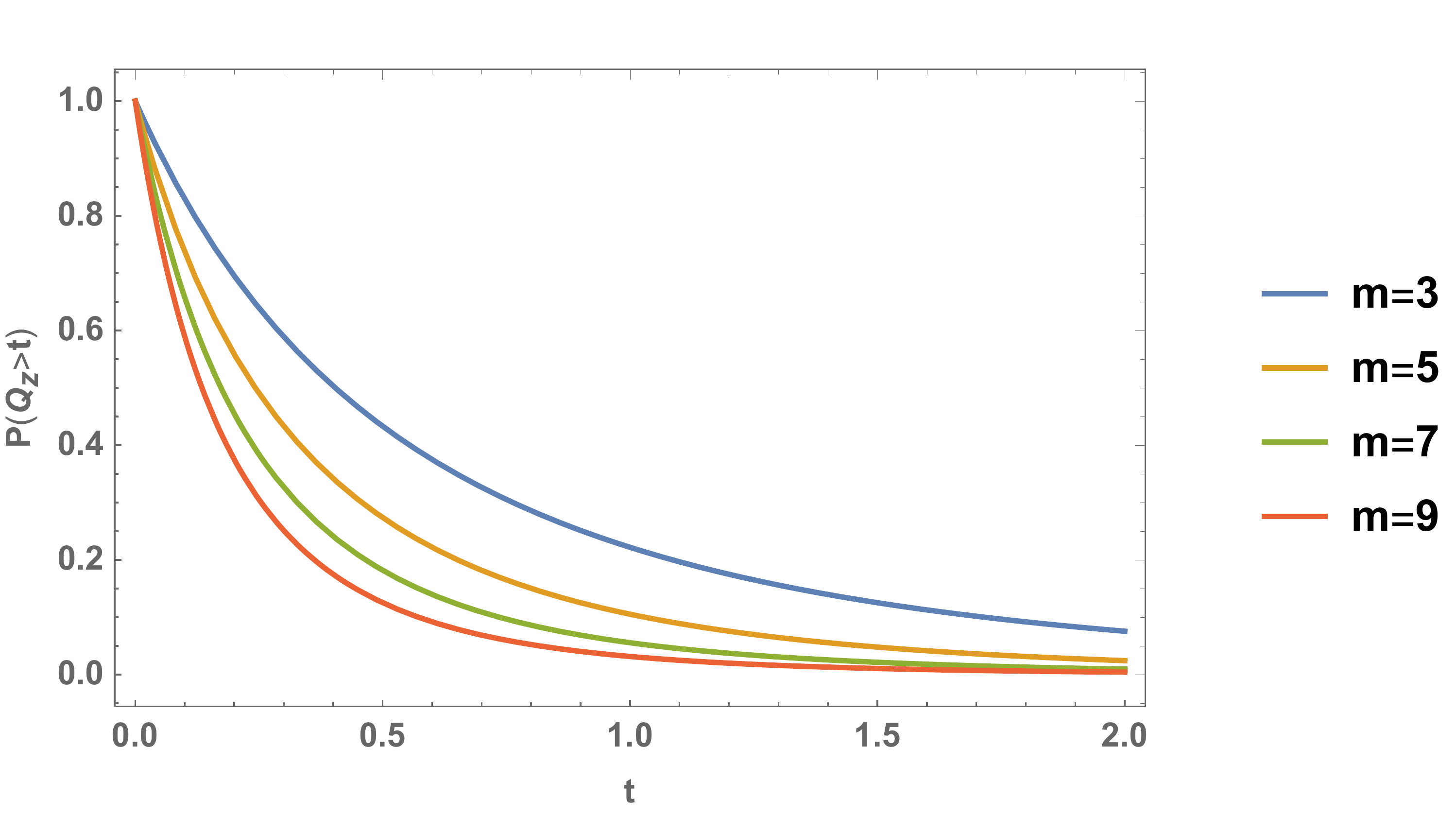}
\caption{Signal interference noise ratio for $m\in\{3,5,7,9\}$, $\lambda=1/100$, $d=1$ at location $z=1$}\label{fig:sinr}
\end{figure}
\FloatBarrier 

\subsection{Economic gravity model} Consider an equispaced grid of locations of countries relative to a country at the origin with spacing $d$. For each location $x$, there exist two independent random variables: a fraction coefficient $U_x$ uniformly taking values in $(0,1)$ and gross domestic product $G_x$, assumed to be $\text{Gamma}(2,\lambda)$ distributed. The gravity model of bilateral trade flows with the origin is described by $T_x=U_xG_x/d_x^2$, where $d_x>0$ is distance. Then the trade flow at $x$ is given by $T_x=U_xG_x\sim\text{Exponential}(\lambda)$ and the total trade flows with the origin is given by \[T = \sum_x \frac{U_xG_x}{(dx)^2}\] which has the Jacobi theta distribution with parameter $m=1/(\lambda d^2)$.

\subsection{Electric field} Consider an equispaced grid of locations of point charges with spacing $d>0$. The point charges are $\{Q_x: x=1,2,\dotsb\}$ and are distributed $\text{Exponential}(4\pi \varepsilon_0 d^2\lambda)$. Then the total charge of the electric field at the origin has the Jacobi theta distribution with parameter $m=1/(4\pi \varepsilon_0 d^2\lambda)$.

\section{Discussion and conclusions}\label{sec:conc}

We describe a continuous univariate distribution supported on the positive reals based on integration of random measures composed of exponential random variables across an inverse-square surface. The Jacobi theta distribution is single-parameter,  positively skewed, and leptokurtic. The cumulative distribution and density functions are expressed in terms of the Jacobi theta function, and asymptotic and log-normal approximations enabling tractable calculations and exact maximum likelihood inference.

\section{Acknowledgements} 

\bibliographystyle{apalike}

\end{document}